\documentclass[preprint,12pt]{elsarticle}

\usepackage{amssymb}
\usepackage{mathrsfs}
\usepackage{amsthm}
\usepackage{amsmath}
\allowdisplaybreaks[1]
\usepackage{subfigure}
\usepackage{color}
\usepackage{longtable,booktabs}
\usepackage{enumerate}
\usepackage{extarrows}
\usepackage{lineno}

\newtheorem{theorem}{Theorem}[section]
\newtheorem{lemma}{Lemma}[section]
\newtheorem{corollary}{Corollary}[section]
\newtheorem{proposition}{Proposition}[section]
\newtheorem{definition}{Definition}[section]
\newtheorem{remark}{Remark}[section]

\begin{document}
%\linenumbers
\begin{frontmatter}



\title{Generalized Ornstein-Uhlenbeck process for affine
stochastic functional differential equations and its applications\tnoteref{title1}}
\tnotetext[title1]{This work was partially supported by the National Natural Science Foundation of China (NSFC) under
Grants No.12171321, No.11971316, No.11771295, No.11501369 and No.11371252; the NSF of Shanghai Grants under No.25ZR1401278, No.19ZR1437100 and No.20JC1413800; Chen Guang Project (14CG43) of Shanghai Municipal Education Commission, Shanghai Education Development Foundation; Yangfan Program of Shanghai (14YF1409100) and Shanghai Gaofeng Project for University Academic Program Development.}
\author{Xiang Lv\corref{cor1}}
\ead{lvxiang@shnu.edu.cn}

\address{Department of Mathematics, Shanghai Normal
University, Shanghai 200234, People's Republic of China}
\cortext[cor1]{Corresponding author}
\begin{abstract}
This paper studies the existence and global stability of generalized Ornstein-Uhlenbeck process for affine stochastic functional differential equations.  Various very basic and important properties are established. In the applications, we present a standard and rigorous procedure for guaranteeing the existence and uniqueness of random equilibria for nonlinear stochastic functional differential equations, which attracts all pull-back trajectories in different types of convergence. Some examples are given to illustrate our main results. The results presented in this paper improve and simplify the conclusions of Jiang and Lv [{\it SIAM J. Control Optim.}, 54 (2016), pp. 2383-2402] and [{\it J. Differential Equations}, 367 (2023), pp. 890-921].
\end{abstract}

\end{frontmatter}


\section{Introduction}
Consider the Ornstein-Uhlenbeck type stochastic differential equation (SDE) in $\mathbb{R}$:
\begin{equation}\label{OU-1}dx(t)=\lambda x(t)dt+\sigma dW(t),
\end{equation}
where $W(t)$ is a two-sided Wiener process in $\mathbb{R}$ and $\lambda$, $\sigma$ are constants. Assume that $\lambda<0$ and $\sigma\neq0$ , it is well known that there exists an exponentially stable $\mathscr{F}_-^W$-measurable stationary solution (random equilibrium)
\begin{equation}\label{OU-2}u(\omega)=\sigma\int_{-\infty}^0e^{-\lambda t}dW(t),\end{equation}
where $\mathscr{F}_-^W=\sigma\{\omega\mapsto W(t,\omega):t\leq0\}$ is the past $\sigma$-algebra, see \cite[Example 2.4.4]{Chu} and \cite[Proposition 3.1]{CS}. It is worth pointing out that the stationary Ornstein-Uhlenbeck process plays a key role on the conjugacy of stochastic and random differential equations, see \cite{DLS,GP,IL,IS}. Based on this observation, we have further investigated the existence and global stability of nontrivial stationary solutions for SDEs with additive or multiplicative white noise, such as \cite{JL1,JL2,Lv}.

In fact, in many real world scenarios, for example, biometrical and ecological systems, finance and economics,
networked systems and wireless communications, time delays are often unavoidable, see Banks \cite{Bank}, Cushing \cite{Cu}, Diekmann et al. \cite{DG}, Franke, H\"{a}rdle and Hafner \cite{FHH}, Hale and Lunel \cite{HL} or Kolmanovskii and Myshkis \cite{KM} for applications. Taking into account the effect of noise, these dynamical systems with memories can be modeled by stochastic functional differential equations, which have been researched extensively and intensively during the past decades, see Caraballo, Garrido-Atienza and Schmalfu{\ss} \cite{CGS}, Caraballo, Real and Taniguchi \cite{CRT},  Liu \cite{L}, Mao \cite{M}, Mohammed \cite{Mo}, Mohammed and Scheutzow \cite{MS1,MS2,MS3} and references therein.

Motivated by the Ornstein-Uhlenbeck process (\ref{OU-1}) and our recent works \cite{JL1,Lv}, we first study the following affine stochastic functional differential equations (SFDEs)
\begin{equation}\label{SFDE}
  dx(t)=L(x_t)dt+\Sigma dB(t),\quad t\geq0,
\end{equation}
\begin{equation}\label{SFDE2}x(s)=x_0(s)=\xi(s),\quad s\in[-\tau,0]\ {\rm and}\ \xi\in C_\tau,
\end{equation}
where $\tau\geq0$, $C_\tau:=C\bigl([-\tau,0],\mathbb{R}^n\bigr)$ denotes the Banach space of continuous functions $\xi:[-\tau,0]\rightarrow\mathbb{R}^n$ equipped with the supremum norm $\|\xi\|_{C_\tau}=\sup_{-\tau\leq s\leq0}\bigl|\xi(s)\bigr|$, $x_t\in C_\tau$ is defined by $x_t(s)=x(t+s)$ for $-\tau\leq s\leq 0$, $L:C_\tau\rightarrow\mathbb{R}^n$ is a bounded linear functional, $\Sigma\in\mathbb{R}^{n\times m}$ and $B(t)=(B_1(t),\ldots,B_m(t))^{\rm T}$ is an $m$-dimensional two-sided Wiener process on the canonical probability space $(\Omega,\mathscr{F},\mathbb{P})$. Here, $\mathscr{F}$ is the Borel $\sigma$-algebra of
$\Omega=C_0(\mathbb{R},\mathbb{R}^m)=\bigl\{\omega=(\omega_1,\omega_2,\ldots,\omega_m)\in C(\mathbb{R},\mathbb{R}^m),\ \omega(0)=0\bigr\}$, which is equipped with the following metric
\[\varrho(\omega,\omega^\ast):=\sum_{k=1}^\infty\frac{1}{2^k}\frac{\varrho_k(\omega,\omega^\ast)}{1+\varrho_k(\omega,\omega^\ast)},\quad
\varrho_k(\omega,\omega^\ast)=\max_{t\in[-k,k]}\bigl|\omega(t)-\omega^\ast(t)\bigr|,\]
and $\mathbb{P}$ is the corresponding Wiener measure. In what follows, let us introduce some basic notations, we define the Euclidean norm $|x|:=\bigl(\sum_{i=1}^n|x_i|^2\bigr)^{1/2}$, $x\in\mathbb{R}^n$, where $\mathbb{R}^n$ is the $n$-dimensional Euclidean space. For any  matrix $A=(A_{ij})_{n\times m}\in\mathbb{R}^{n\times m}$, let $\|A\|:=\bigl(\sum_{i=1}^n\sum_{j=1}^m|A_{ij}|^2\bigr)^{1/2}$, where $\mathbb{R}^{n\times m}$ denotes the space of all $n\times m$-dimensional matrices with real entries. $\mathbb{C}$ stands for the set of complex numbers and ${\rm Re}(z)$ denotes the real part of $z\in\mathbb{C}$.

Observe that $L$ is a bounded linear mapping from $C_\tau$ to $\mathbb{R}^n$ with the operator topology, then the Riesz representation theorem shows that there exists an $\mathbb{R}^{n\times n}$-valued measure $\mu=(\mu_{ij})_{n\times n}\in\mathscr{M}([-\tau,0],\mathbb{R}^{n\times n})$ such that
\begin{equation}\label{Riesz}
L(\phi)=\int_{-\tau}^{0}\mu(ds)\phi(s)
\end{equation}
for all $\phi\in C_\tau$, where $\mathscr{M}([-\tau,0],\mathbb{R}^{n\times n})$ denotes the space of finite signed Borel measures on $[-\tau,0]$ with values in $\mathbb{R}^{n\times n}$. Let $\mathscr{B}$ be the Borel $\sigma$-algebra on $[-\tau,0]$, the measure $|\mu|$ is called the total variation measure, and the function $|\mu|:\mathscr{B}\rightarrow\mathbb{R}_+$ is defined by
\[|\mu|(A)\triangleq\sup\sum_{i=1}^{N}|\mu(A_i)|,\]
where $\{A_i\}_{i=1}^{N}\subset\mathscr{B}$ is a partition of $A\in\mathscr{B}$ and the supremum is taken over all partitions of $A$.

With the help of (\ref{Riesz}), we can consider the linear part of (\ref{SFDE}) and (\ref{SFDE2})
\begin{equation}\label{LSFDE}
  dy(t)=\int_{-\tau}^{0}\mu(ds)y(t+s),\quad t\geq0,
\end{equation}
\begin{equation}\label{LSFDE2}y(s)=y_0(s)=\xi(s),\quad s\in[-\tau,0]\ {\rm and}\ \xi\in C_\tau,
\end{equation}
which is a deterministic delay equation. It is well known that the unique solution of (\ref{LSFDE}) and (\ref{LSFDE2}) denoted by $y(t,\xi)$ has the following form
\begin{equation}\label{LSFDE3}
y(t,\xi)=r(t)\xi(0)+\int_{-\tau}^0\int_{-\tau}^{u}r(t+s-u)\mu(ds)\xi(u)du, \quad t\geq0,
\end{equation}
see Theorem 1.2 in \cite[Chapter 6]{HL}, where $r(t)$ is the fundamental solution of (\ref{LSFDE}) with values $r(0)=\mathbf{I}_{n\times n}$ and $r(t)=\mathbf{0}_{n\times n}$ for $t\in[-\tau,0)$. Here, $\mathbf{I}_{n\times n}$  is the $n\times n$-dimensional identity matrix and $\mathbf{0}_{n\times n}$  is the $n\times n$-dimensional zero matrix.
Let
\begin{equation}\label{Eigen}
\alpha_0\triangleq\sup\left\{{\rm Re}(\lambda):\lambda\in\mathbb{C},\det\left(\lambda\mathbf{I}_{n\times n}-\int_{-\tau}^0e^{\lambda s}\mu(ds)\right)=0\right\},
\end{equation}
where $\det(A)$ denotes the determinant of an  $n\times n$-dimensional matrix $A$. Furthermore, by Theorem 3.2 in \cite[Chapter 9]{HL}, for any $\alpha<-\alpha_0$, there exists a constant $C_\alpha$ such that
\begin{equation}\label{Eigen2}
\|r(t)\|\leq C_\alpha e^{-\alpha t},\quad t\geq-\tau.
\end{equation}
Therefore, combining (\ref{LSFDE3}) and the variation-of-constants formula, see Lemma 6.1 in \cite{RRG}, the solution of $\{x(t,\xi):t\geq-\tau\}$ of (\ref{SFDE}) and (\ref{SFDE2}) can be written as the following form
\begin{align}\label{VCF}
x(t,\xi)&=y(t,\xi)+\int_0^tr(t-s)\Sigma dB(s)\nonumber\\
&=r(t)\xi(0)+\int_{-\tau}^0\int_{-\tau}^{u}r(t+s-u)\mu(ds)\xi(u)du+\int_0^tr(t-s)\Sigma dB(s)
\end{align}
for $t\geq0$ and $x(t,\xi)=\xi(t)$ for $t\in[-\tau,0]$, where $r(t)$ is the resolvent of (\ref{LSFDE}) and (\ref{LSFDE2}).

The rest of the paper is organized as follows. Section 2 is devoted to the existence and global stability of generalized Ornstein-Uhlenbeck process for affine
SFDEs, and some basic and important properties are demonstrated. Section 3 gives the proof of main results presented in Section 2. Section 4 provides a rigorous and standard procedure for guaranteeing the global stability of nonlinear SFDEs with the help of generalized Ornstein-Uhlenbeck process. Section 5 gives the proof of main results presented in Section 4. Section 6 gives some examples to illustrate our main results. Section 7 ends this paper with some discussions and open problems.

\section{Generalized Ornstein-Uhlenbeck process of affine SFDEs}
In this section, we mainly consider the long-term behaviour of pullback trajectories for (\ref{SFDE}) and (\ref{SFDE2}) and some related properties. To this end, we need to show some definitions with respect to the theory of random dynamical systems (RDSs), and then give our main results at the end of this section. The reader is referred to \cite{A,Chu} for more details.
Let $X$ be a complete separable metric space (i.e., Polish space) equipped with the Borel $\sigma$-algebra $\mathscr{B}(X)$ and $(\Omega,\mathscr{F},\mathbb{P})$ be a probability space.

\begin{definition}
$\theta\equiv\bigl(\Omega,\mathscr{F},\mathbb{P},\{\theta_t,t\in\mathbb{R}\}\bigr)$ is called a metric dynamical system (MDS) if $\theta$ is a measurable flow:
\[\theta:\mathbb{R}\times\Omega\mapsto\Omega,\qquad \theta_0={\rm id},\qquad \theta_{t_2}\circ\theta_{t_1}=\theta_{t_1+t_2}\]
for all $t_1,t_2\in\mathbb{R}$, which is $\bigl(\mathscr{B}(\mathbb{R})\otimes\mathscr{F}, \mathscr{F}\bigr)$-measurable. Moreover, we assume that $\theta_t\mathbb{P}=\mathbb{P}$ for all $t\in\mathbb{R}$.
\end{definition}

\begin{definition}\label{DRDS}
An RDS on the state space $X$ with an MDS $\theta$ is a mapping
\[\varphi:\mathbb{R}_+\times\Omega\times X\mapsto X, \quad (t,\omega,x)\mapsto\varphi(t,\omega,x),\]
which is $\bigl(\mathscr{B}(\mathbb{R}_+)\otimes\mathscr{F}\otimes\mathscr{B}(X),
\mathscr{B}(X)\bigr)$-measurable such that for all $\omega\in\Omega$,
\begin{enumerate}[{\rm(i)}]
\item $\varphi(0,\omega,\cdot)$ is the identity on $X$;
\item $\varphi(t_1+t_2,\omega,x)=\varphi\bigl(t_2,\theta_{t_1}\omega,\varphi(t_1,\omega,x)\bigr)$ for all $t_1,t_2\in\mathbb{R}_+$ and $x\in X$;
\item the mapping $\varphi(t,\omega,\cdot):X\to X$ is continuous for all $t\in
\mathbb{R}_+$.
\end{enumerate}
\end{definition}

\begin{definition}\label{Tem}
A random variable $R(\omega)$ is said to be tempered with respect to the MDS $\theta\equiv\bigl(\Omega,\mathscr{F},\mathbb{P},\{\theta_t,t\in\mathbb{R}\}\bigr)$ if
\[\sup_{t\in\mathbb{R}}\bigl\{e^{-\gamma|t|}|R(\theta_{t}\omega)|\bigr\}<\infty\quad for\ \ all\ \ \omega\in\Omega\ \ and\ \ \gamma>0.\]
\end{definition}

\begin{definition}\label{Equili}
A random variable $v:\Omega\rightarrow X$ is said to be a random equilibrium (or stationary solution) of the RDS $(\theta,\varphi)$ if it is invariant with respect to  $(\theta,\varphi)$, i.e.
\[\varphi(t,\omega)v(\omega)=v(\theta_t\omega)\quad for\ \ all\ \ t\geq0\ \ and\ \ \omega\in\Omega.\]
\end{definition}

Motivated by \cite{JL1,JL3,KMen}, this section is devoted to the global stability of stochastic flows generated by (\ref{SFDE}) and (\ref{SFDE2}). Define $\varphi(t,\omega,\xi)=x_t(\omega,\xi)$ denote the segment process of the unique solution for (\ref{SFDE}) and (\ref{SFDE2}), which generates an RDS in the state space $C_\tau$, see Theorem 2.9 in \cite{BJX} or \cite{JL3}, we can have the following theorems.

\begin{theorem}\label{thm1} Assume that $\alpha_0<0$, then for any $\xi\in C_\tau$ and $\omega\in\Omega$,
\begin{equation}\label{Conclu1}
\lim_{t\rightarrow\infty}\varphi(t,\theta_{-t}\omega,\xi)(\bullet)=\int_{-\infty}^{\bullet}r(\bullet-u)\Sigma dB(u)
\end{equation}
in $C_\tau$. Furthermore, set
\begin{equation}\label{Conclu2}U(\omega)(s)=\int_{-\infty}^{s}r(s-u)\Sigma dB(u)\quad {\rm for}\quad-\tau\leq s\leq0,\end{equation}
it follows that the random variable $U(\omega)$ with values in $C_\tau$ is a random equilibrium, i.e.
\begin{equation}\label{Conclu3}\varphi\bigl(t,\omega,U(\omega)\bigr)=U(\theta_{t}\omega)\end{equation}
for all $t\geq0$ and $\omega\in\Omega$.
\end{theorem}

\begin{remark}
{\rm For any $-\tau\leq s\leq0$, the random variable $U(\omega)(s)$ is well defined. Note that $\alpha_0<0$, fix $s\in[-\tau,0]$, using the same procedure in \cite{JL1} and Doob's martingale convergence theorem, see {\rm \cite[Problem 3.20 in Chap. 1]{KS}},  we conclude that $\int_{-t}^{s}r(s-u)\Sigma dB(u)$ will convergent to $U(\omega)(s)$ as $t\rightarrow\infty$ for $\omega\in\Omega$.
}
\end{remark}

\begin{remark}
{\rm In the proof of Theorem \ref{thm1}, we need to use the law of iterated logarithm, i.e.
\[
\limsup_{t\rightarrow\pm\infty}\frac{|B(t)|}{\sqrt{2|t|\log\log|t|}}=1\quad \mathbb{P}\mbox{-a.s.},
\]
which together with the definition of $\theta$ implies that
\begin{equation}\label{Law}\Omega^*=\left\{\omega\Bigg|\limsup_{t\rightarrow\pm\infty}\frac{|B(t,\omega)|}{\sqrt{2|t|\log\log|t|}}=1,\ \omega\in\Omega\right\}\end{equation}
is a $\theta$-invariant set of full measure. From Remark 1.2.1 in \cite{Chu}, we can choose an indistinguishable form of $\varphi$ which coincides on the $\theta$-invariant set $\Omega^*$. In the rest of this paper, by a slight abuse of notation, we still write $\Omega$ instead of $\Omega^*$.
}
\end{remark}

\begin{remark}
{\rm In Theorem \ref{thm1}, we prove that if the spectral radius of characteristic equations is negative, all pullback trajectories will convergent to the globally stable stationary solution $U(\omega)$ in the state space $C_\tau$. Moreover, the process  $U(\theta_{t}\omega)$ can be regarded as the generalized Ornstein-Uhlenbeck process for affine SFDEs, which includes the corresponding results (in the sense of distributions) in \cite{KMen}.  In fact, the bounded linear functional $L$ given in \cite{KMen} is one dimensional and has a special form
\[L(x_t)=ax(t)+bx(t-\tau),\ a\in\mathbb{R},\ b\in\mathbb{R}.\]
}
\end{remark}

\begin{proposition}\label{prop1}  Assume that $\alpha_0<0$, let $U(\omega)$ be defined in Theorem \ref{thm1}, we have the following properties:
\begin{enumerate}[{\rm(i)}]
\item The random variable $U:\Omega\rightarrow C_\tau$ is $\mathscr{F}_-$-measurable and
\begin{equation}\label{OU1}
\lim_{t\rightarrow\pm\infty}\frac{\|U(\theta_t\omega)\|_{C_\tau}}{|t|}=0\quad {\rm for}\quad \omega\in\Omega,\end{equation}
where $\mathscr{F}_-=\sigma\bigl\{\omega\mapsto B(t,\omega):t\leq0\bigr\}$ is the past $\sigma$-algebra;
\item $t\rightarrow U(\theta_t\omega)$ is continuous from $\mathbb{R}$ to $C_\tau$ for all $\omega\in\Omega$ and the random variable $U:\Omega\rightarrow C_\tau$ is tempered, i.e.
   \begin{equation}\label{OU2}\sup_{t\in\mathbb{R}}\bigl\{e^{-\gamma|t|}\|U(\theta_t\omega)\|_{C_\tau}\bigr\}<\infty\quad for\ \ all\ \ \omega\in\Omega\ \ and\ \ \gamma>0;\end{equation}
\item The process
\begin{equation}\label{OU3}
\varphi\bigl(t,\omega,U(\omega)\bigr)(0)=U(\theta_{t}\omega)(0)=\int_{-\infty}^{t}r(t-u)\Sigma dB(u),\ t\geq0\end{equation}
is the solution of (\ref{SFDE}) with the initial value $U(\omega)$, i.e.
\begin{equation}\label{OU3-2}
  d[U(\theta_{t}\omega)(0)]=L\bigl(U(\theta_{t}\omega)\bigr)dt+\Sigma dB(t),\quad t\geq0;
\end{equation}
\item For any $-\tau\leq s\leq0$, $\mathbb{E}[U(\omega)(s)]=0$,
\begin{equation}\label{OU4}
\mathbb{E}|U(\omega)(s)|^2=\int_{-\infty}^{s}\left|r(s-u)\Sigma\right|^2du=\int_{-\infty}^{0}\left|r(-u)\Sigma\right|^2du\leq\frac{C_\alpha^2\|\Sigma\|^2}{2\alpha}
\end{equation}
and then
\begin{equation}\label{OU5}
\mathbb{E}|U(\omega)(s)|\leq\left(\mathbb{E}|U(\omega)(s)|^2\right)^\frac12\leq\frac{C_\alpha\|\Sigma\|}{\sqrt{2\alpha}},
\end{equation}
where $\alpha\in(0,-\alpha_0)$ and $\mathbb{E}|U(\omega)(s)|^2$ is independent of $s$;
\item
\begin{equation}\label{OU6}
\mathbb{E}\|U(\omega)\|_{C_\tau}\leq2m\sqrt\tau\|\Sigma\|+\sqrt{\frac{2}{\pi\alpha^3}}me^{\alpha\tau}C_\alpha\|\Sigma\|\Gamma\left(\frac32\right)\int_{-\tau}^0e^{-\alpha\rho}|\mu|(d\rho)
\end{equation}
and
\begin{equation}\label{OU7}
\mathbb{E}\|U(\omega)\|_{C_\tau}^2\leq8m\tau\|\Sigma\|^2+\frac{2m}{\alpha^3}e^{2\alpha\tau}C_\alpha^2\|\Sigma\|^2\left|\int_{-\tau}^0e^{-\alpha\rho}|\mu|(d\rho)\right|^2.
\end{equation}
\end{enumerate}
\end{proposition}

\begin{theorem}\label{thm1-1} Assume that $\alpha_0<0$, then for any $\xi\in C_\tau$ and $\omega\in\Omega$,
\begin{equation}\label{Conclu1-1}
\lim_{t\rightarrow\infty}\mathbb{E}\|\varphi(t,\theta_{-t}\omega,\xi)-U(\omega)\|^2_{C_\tau}
=0,
\end{equation}
where $U(\omega)$ is given in (\ref{Conclu2}).
\end{theorem}

\begin{corollary}\label{cor1} Assume that $\alpha_0<0$, then for any $\xi\in C_\tau$ and $\omega\in\Omega$,
\begin{equation}\label{Conclu1-2}
\lim_{t\rightarrow\infty}\mathbb{E}\|\varphi(t,\theta_{-t}\omega,\xi)-U(\omega)\|_{C_\tau}
=0,
\end{equation}
where $U(\omega)$ is given in (\ref{Conclu2}).
\end{corollary}

\section{Proofs of main results in Section 2}
\noindent
{\bf Proof of Theorem \ref{thm1}.}
First, by (\ref{VCF}) and the definition of wiener shift $\theta$, we can easily have that for all $t\geq\tau$ and $-\tau\leq s\leq0$
\begin{align}\label{eq1}
&\varphi(t,\omega,\xi)(s)\nonumber\\
=&r(t+s)\xi(0)+\int_{-\tau}^0\int_{-\tau}^{u}r(t+s+\rho-u)\mu(d\rho)\xi(u)du+\int_0^{t+s}r(t+s-\rho)\Sigma dB(\rho),
\end{align}
and then
\begin{align}\label{eq2}
&\varphi(t,\theta_{-t}\omega,\xi)(s)\nonumber\\
=&r(t+s)\xi(0)+\int_{-\tau}^0\int_{-\tau}^{u}r(t+s+\rho-u)\mu(d\rho)\xi(u)du+\int_{-t}^sr(s-\rho)\Sigma dB(\rho).
\end{align}
Since $\alpha_0<0$, for any $\alpha\in(0,-\alpha_0)$, from (\ref{Eigen2}), it follows that
\begin{equation}\label{Eigen3}
\|r(t)\|\leq C_\alpha e^{-\alpha t},\quad t\geq-\tau,
\end{equation}
where $C_\alpha$ is a positive constant. Note that $t+s+\rho-u\geq-\tau$ for $t\geq\tau$, $-\tau\leq s\leq0$ and $-\tau\leq\rho\leq\rho-u\leq0$, using (\ref{Eigen3}), it is clear that
\begin{equation}\label{eq3}
\sup_{-\tau\leq s\leq0}|r(t+s)\xi(0)|\leq e^{\alpha\tau}C_\alpha e^{-\alpha t}\|\xi\|_{C_\tau}
\end{equation}
and
\begin{align}\label{eq4}
&\sup_{-\tau\leq s\leq0}\left|\int_{-\tau}^0\int_{-\tau}^{u}r(t+s+\rho-u)\mu(d\rho)\xi(u)du\right|\nonumber\\
\leq&e^{2\alpha\tau}C_\alpha e^{-\alpha t}\int_{-\tau}^0\int_{-\tau}^{u}|\mu|(d\rho)|\xi(u)|du\nonumber\\
\leq&\tau e^{2\alpha\tau}C_\alpha e^{-\alpha t}\|\xi\|_{C_\tau}\cdot|\mu|\bigl([-\tau,0]\bigr),
\end{align}
which implies that
\begin{equation}\label{eq4-1}
\sup_{-\tau\leq s\leq0}|r(t+s)\xi(0)|\rightarrow0\quad {\rm as}\quad t\rightarrow\infty
\end{equation}
and
\begin{equation}\label{eq5}
\sup_{-\tau\leq s\leq0}\left|\int_{-\tau}^0\int_{-\tau}^{u}r(t+s+\rho-u)\mu(d\rho)\xi(u)du\right|\rightarrow0\quad {\rm as}\quad t\rightarrow\infty.
\end{equation}
Combining (\ref{eq2}), (\ref{eq4-1}) and (\ref{eq5}), in order to prove (\ref{Conclu1}), we only need to show that
\begin{equation}\label{eq6}
\lim_{t\rightarrow\infty}\int_{-t}^\bullet r(\bullet-u)\Sigma dB(u)=\int_{-\infty}^{\bullet}r(\bullet-u)\Sigma dB(u)\quad {\rm in}\quad C_\tau.
\end{equation}
Note that the resolvent of  (\ref{LSFDE}) and (\ref{LSFDE2}), i.e. $\{r(t):t\geq0\}$ is a locally absolutely continuous function satisfying the following equation
\begin{align}\label{eq7}
r(t)&=\mathbf{I}_{n\times n}+\int_0^t\int_{[\max\{-\tau,-s\},0]}\mu(du)r(s+u)ds\nonumber\\
&=\mathbf{I}_{n\times n}+\int_0^t\int_{-s}^0\mu(du)r(s+u)ds\quad {\rm for}\quad0\leq t\leq\tau
\end{align}
and
\begin{equation}\label{eq8}
r(t)=\mathbf{I}_{n\times n}+\int_0^\tau\int_{-s}^0\mu(du)r(s+u)ds+\int_\tau^t\int_{-\tau}^0\mu(du)r(s+u)ds
\end{equation}
for all $t\geq\tau$, see \cite{AMW}. This yields that
\begin{equation}\label{eq9}
r'(t)=\int_{-\tau}^0\mu(du)r(t+u)\quad {\rm for}\quad t\geq0.
\end{equation}
Consequently, using the integration by parts formula, for any $t\geq\tau$ and $-\tau\leq s\leq0$, we can see that
\begin{align}\label{eq10}
&\int_{-t}^s r(s-u)\Sigma dB(u)-\int_{-\infty}^{s}r(s-u)\Sigma dB(u)\nonumber\\
=&\int_{-\infty}^{-t} r(s-u)\Sigma dB(u)\nonumber\\
=&r(t+s)\Sigma B(-t)-\int_{-\infty}^{-t}\frac{\rm d}{{\rm d}u}\bigl[r(s-u)\bigr]\Sigma B(u)du\nonumber\\
=&r(t+s)\Sigma B(-t)+\int_{-\infty}^{-t}\left[\int_{-\tau}^0\mu(d\rho)r(s-u+\rho)\right]\Sigma B(u)du,
\end{align}
where we use the well-known law of iterated logarithm, (\ref{Eigen3}) and (\ref{eq9}).
From (\ref{Eigen3}) and (\ref{eq10}), for any $t\geq\tau$ and $\alpha\in(0,-\alpha_0)$, it is immediate that
\begin{align}\label{eq10-1}
&\sup_{-\tau\leq s\leq0}\left|\int_{-t}^s r(s-u)\Sigma dB(u)-\int_{-\infty}^{s}r(s-u)\Sigma dB(u)\right|\nonumber\\
\leq&\sup_{-\tau\leq s\leq0}\left|r(t+s)\Sigma B(-t)\right|+\sup_{-\tau\leq s\leq0}\left|\int_{-\infty}^{-t}\left[\int_{-\tau}^0\mu(d\rho)r(s-u+\rho)\right]\Sigma B(u)du\right|\nonumber\\
\leq&e^{\alpha\tau}C_\alpha e^{-\alpha t}\|\Sigma\|\cdot|B(-t)|\nonumber\\
\quad&+e^{2\alpha\tau}C_\alpha\|\Sigma\|\cdot|\mu|([-\tau,0])\sup_{u\leq0}\left\{e^{\frac\alpha2u}|B(u)|\right\}\cdot\int_{-\infty}^{-t}e^{\frac\alpha2u}du,
\end{align}
which together with the law of iterated logarithm induces that
\[\sup_{-\tau\leq s\leq0}\left|\int_{-t}^s r(s-u)\Sigma dB(u)-\int_{-\infty}^{s}r(s-u)\Sigma dB(u)\right|\longrightarrow0\quad {\rm as}\quad t\rightarrow\infty.\]
That is, (\ref{Conclu1}) holds. Finally, by (\ref{Conclu1}) and the continuity and cocycle property for $(\theta,\varphi)$, we can deduce that for any $t\geq0$ and $\xi\in C_\tau$,
\begin{align}\label{eq10-2}
&\varphi\bigl(t,\omega,U(\omega)\bigr)\nonumber\\
=&\varphi\bigl(t,\omega,\lim_{\tilde{t}\rightarrow\infty}\varphi(\tilde{t},\theta_{-\tilde{t}}\omega,\xi)\bigr)\qquad\qquad\qquad\,{\rm by\ (\ref{Conclu1})}\nonumber\\
=&\lim_{\tilde{t}\rightarrow\infty}\varphi\bigl(t,\omega,\varphi(\tilde{t},\theta_{-\tilde{t}}\omega,\xi)\bigr)\qquad\qquad\qquad{\rm by\ continuity}\nonumber\\
=&\lim_{\tilde{t}\rightarrow\infty}\varphi(t+\tilde{t},\theta_{-\tilde{t}}\omega,\xi)\qquad\qquad\qquad\qquad{\rm by\ cocycle}\nonumber\\
=&\lim_{\tilde{t}\rightarrow\infty}\varphi(t+\tilde{t},\theta_{-(t+\tilde{t})}\circ\theta_t\omega,\xi)\nonumber\\
=&U(\theta_t\omega),\quad t\geq0,\ \omega\in\Omega,
\end{align}
which gives (\ref{Conclu3}). The proof is complete.
\quad\quad$\Box$

\noindent
{\bf Proof of Proposition \ref{prop1}.} The proof will be divided into five parts.

Firstly, for any $-\tau\leq s\leq0$, it is evident that $U(\cdot)(s):\Omega\rightarrow \mathbb{R}^{n}$ is $\mathscr{F}_-$-measurable, which together with Lemma II.2.1 in \cite{Mo} yields that $U:\Omega\rightarrow C_\tau$ is $\mathscr{F}_-$-measurable. Thanks to the integration by parts formula, the law of iterated logarithm, (\ref{Eigen3}) and (\ref{eq9}), we can write
\begin{align}\label{eq11}
&U(\omega)(s)\nonumber\\
=&\int_{-\infty}^{s}r(s-u)\Sigma dB(u)\nonumber\\
=&\Sigma B(s,\omega)-\int_{-\infty}^{s}\frac{\rm d}{{\rm d}u}\bigl[r(s-u)\bigr]\Sigma B(u,\omega)du\nonumber\\
=&\Sigma B(s,\omega)+\int_{-\infty}^{s}\left[\int_{-\tau}^0\mu(d\rho)r(s-u+\rho)\right]\Sigma B(u,\omega)du
\end{align}
for $-\tau\leq s\leq0$ and $\omega\in\Omega$, and then
\begin{align}\label{eq12}
&U(\theta_t\omega)(s)\nonumber\\
=&\Sigma B(s,\theta_t\omega)+\int_{-\infty}^{s}\left[\int_{-\tau}^0\mu(d\rho)r(s-u+\rho)\right]\Sigma B(u,\theta_t\omega)du\nonumber\\
=&\Sigma B(t+s,\omega)-\Sigma B(t,\omega)+\int_{-\infty}^{s}\left[\int_{-\tau}^0\mu(d\rho)r(s-u+\rho)\right]\Sigma B(t+u,\omega)du\nonumber\\
&-\int_{-\infty}^{s}\left[\int_{-\tau}^0\mu(d\rho)r(s-u+\rho)\right]\Sigma B(t,\omega)du
\end{align}
for all $t\geq0$. This shows that
\begin{align}\label{eq13}
&\left\|U(\theta_t\omega)\right\|_{C_\tau}\nonumber\\
=&\sup_{-\tau\leq s\leq0}\left|U(\theta_t\omega)(s)\right|\nonumber\\
\leq&\|\Sigma\|\sup_{-\tau\leq s\leq0}|B(t+s,\omega)|+\|\Sigma\|\cdot|B(t,\omega)|\nonumber\\
&+\sup_{-\tau\leq s\leq0}\int_{-\infty}^{s}\left\|\int_{-\tau}^0\mu(d\rho)r(s-u+\rho)\right\|\cdot\|\Sigma\|\cdot|B(t+u,\omega)|du\nonumber\\
&+\left(\sup_{-\tau\leq s\leq0}\int_{-\infty}^{s}\left\|\int_{-\tau}^0\mu(d\rho)r(s-u+\rho)\right\|\cdot\|\Sigma\|du\right)|B(t,\omega)|\nonumber\\
\triangleq&J_1(t,\omega)+J_2(t,\omega)+J_3(t,\omega)+J_4(t,\omega).
\end{align}
In order to prove (\ref{OU1}), it is sufficient to verify
\begin{equation}\label{eq14}
\lim_{t\rightarrow\pm\infty}\frac{J_i(t,\omega)}{|t|}=0,\quad 1\leq i\leq4.
\end{equation}
For $i=1$, by (\ref{Law}) and the fact that $-\tau\leq s\leq0$, it is obvious that (\ref{eq14}) is true. The same conclusion holds for cases $i=2,4$.
Next, from (\ref{Law}) and the continuity of $B(t,\omega)$, we can check that for $\omega\in\Omega$ and $\frac12<\epsilon<1$, there exists a constant $C_\epsilon^\omega$ such that
\begin{equation}\label{eq15}
|B(t+u,\omega)|\leq C_\epsilon^\omega+|t+u|^\epsilon\leq C_\epsilon^\omega+|t|^\epsilon+|u|^\epsilon,\ u\leq0,
\end{equation}
which together with (\ref{Eigen3}) implies that
\begin{align}\label{eq16}
&\lim_{t\rightarrow\pm\infty}\frac{J_3(t,\omega)}{|t|}\nonumber\\
\leq&\lim_{t\rightarrow\pm\infty}\frac{e^{2\alpha\tau}C_\alpha\|\Sigma\|\cdot|\mu|([-\tau,0])}{|t|}\int_{-\infty}^{0}e^{\alpha u}(C_\epsilon^\omega+|t|^\epsilon+|u|^\epsilon)du\nonumber\\
=&0.
\end{align}
This gives the limit result for $i=3$ in (\ref{eq14}). That is, (\ref{OU1}) is correct.

Secondly, fix $t_0\in\mathbb{R}$ and $\omega\in\Omega$, for every sequence $\{t_n\}_{n=1}^\infty$
such that $\lim_{n\rightarrow\infty}t_n=t_0$, it is easy to see that the set $\{t_n+s|-\tau\leq s\leq0,\ n\geq0\}$ is bounded. In addition, by (\ref{Eigen3}), we conclude that
\begin{align}\label{eq17}
&\sup_{-\tau\leq s\leq0}\int_{-\infty}^{s}\left\|\int_{-\tau}^0\mu(d\rho)r(s-u+\rho)\right\|\cdot\left|\bigl(B(t_n+u,\omega)-B(t_0+u,\omega)\bigr)\right|du\nonumber\\
\leq&e^{2\alpha\tau}C_\alpha|\mu|([-\tau,0])\int_{-\infty}^{0}e^{\alpha u}\left|\bigl(B(t_n+u,\omega)-B(t_0+u,\omega)\bigr)\right|du,
\end{align}
where $\alpha\in(0,-\alpha_0)$. Combining (\ref{eq12}), (\ref{eq17}) and the continuity of $B(t,\omega)$, for the continuity of $U(\theta_t\omega)$ in $C_\tau$ with respect to $t$, it suffices to show
\[
\lim_{n\rightarrow\infty}\int_{-\infty}^{0}e^{\alpha u}\left|\bigl(B(t_n+u,\omega)-B(t_0+u,\omega)\bigr)\right|du=0,
\]
which can be done by the law of iterated logarithm and Lebesgue's dominated convergence theorem. Therefore, due to the continuity of $U(\theta_t\omega)$ and (\ref{OU1}), the temperedness of $U$ (\ref{OU2}) is evident.

Thirdly, the proof of (iii) can be obtained directly from (\ref{Conclu3}) and the definition of $\theta$.

Fourthly, for all $-\tau\leq s\leq0$, set
\[M^s=\left\{M_t^s=\int_{-t}^{s}r(s-u)\Sigma dB(u):t\geq-s\right\},\]
it is immediate that the process $M^s$ is a martingale and then
\[\lim_{t\rightarrow\infty}\int_{-t}^{s}r(s-u)\Sigma dB(u)=\int_{-\infty}^{s}r(s-u)\Sigma dB(u)\quad {\rm in}\quad L^2(\Omega,\mathscr{F},\mathbb{P}),\]
which is based on (\ref{Eigen3}) and Doob's martingale convergence theorem, see {\rm \cite[Problem 3.20 in Chap. 1]{KS}}. Consequently, $\mathbb{E}[U(\omega)(s)]=0$,
\begin{align}\label{eq18}
\mathbb{E}|U(\omega)(s)|^2=&\int_{-\infty}^{s}\left|r(s-u)\Sigma\right|^2du\nonumber\\
=&\int_{-\infty}^{0}\left|r(-u)\Sigma\right|^2du\nonumber\\
\leq&C_\alpha^2\|\Sigma\|^2\int_{-\infty}^{0}e^{2\alpha u}du\nonumber\\
=&\frac{C_\alpha^2\|\Sigma\|^2}{2\alpha},\quad \alpha\in(0,-\alpha_0),
\end{align}
which gives (\ref{OU4}). From (\ref{eq18}) and H\"{o}lder's inequality, (\ref{OU5}) is proved.

Fifthly, note that the process $\{\int_{-\infty}^{s}r(s-u)\Sigma dB(u):-\tau\leq s\leq0\}$ is not a martingale. For $\alpha\in(0,-\alpha_0)$, using (\ref{eq11}), the Fubini theorem, Doob's martingale inequality and H\"{o}lder's inequality, it is easily seen that
\begin{align}\label{eq19}
\mathbb{E}\|U(\omega)\|_{C_\tau}\leq&\|\Sigma\|\cdot\mathbb{E}\sup_{-\tau\leq s\leq0}|B(s,\omega)|\nonumber\\
&+\mathbb{E}\left(\sup_{-\tau\leq s\leq0}\int_{-\infty}^{s}\left|\int_{-\tau}^0\mu(d\rho)r(s-u+\rho)\right|\cdot\|\Sigma\|\cdot|B(u,\omega)|du\right)\nonumber\\
\leq&\|\Sigma\|\sum_{i=1}^m\left(\mathbb{E}\sup_{-\tau\leq s\leq0}|B_i(s,\omega)|\right)\nonumber\\
&+e^{\alpha\tau}C_\alpha\|\Sigma\|\int_{-\tau}^0e^{-\alpha\rho}|\mu|(d\rho)\int_{-\infty}^0e^{\alpha u}\mathbb{E}|B(u,\omega)|du\nonumber\\
\leq&\|\Sigma\|\sum_{i=1}^m\left(\mathbb{E}\sup_{-\tau\leq s\leq0}|B_i(s,\omega)|^2\right)^\frac12\nonumber\\
&+e^{\alpha\tau}C_\alpha\|\Sigma\|\int_{-\tau}^0e^{-\alpha\rho}|\mu|(d\rho)\int_{-\infty}^0e^{\alpha u}\sum_{i=1}^m\mathbb{E}|B_i(u,\omega)|du\nonumber\\
\leq&\|\Sigma\|\sum_{i=1}^m\left(4\mathbb{E}|B_i(-\tau,\omega)|^2\right)^\frac12\nonumber\\
&+me^{\alpha\tau}C_\alpha\|\Sigma\|\int_{-\tau}^0e^{-\alpha\rho}|\mu|(d\rho)\int_{-\infty}^0e^{\alpha u}\sqrt{\frac{-2u}{\pi}}du\nonumber\\
=&2m\sqrt\tau\|\Sigma\|+\sqrt{\frac{2}{\pi\alpha^3}}me^{\alpha\tau}C_\alpha\|\Sigma\|\Gamma\left(\frac32\right)\int_{-\tau}^0e^{-\alpha\rho}|\mu|(d\rho),
\end{align}
where
\[\mathbb{E}|B_i(u,\omega)|=\int_{-\infty}^{\infty}\frac{|x|}{\sqrt{-2\pi u}}e^{\frac{x^2}{2u}}dx=2\int_{0}^{\infty}\frac{|x|}{\sqrt{-2\pi u}}e^{\frac{x^2}{2u}}dx=\sqrt{\frac{-2u}{\pi}}\]
and
\[\int_{-\infty}^0e^{\alpha u}\sqrt{\frac{-2u}{\pi}}du=\sqrt{\frac{2}{\pi\alpha^3}}\int_0^{\infty}e^{-v}v^{\frac12}dv=\sqrt{\frac{2}{\pi\alpha^3}}\Gamma\left(\frac32\right).\]
Here, $\Gamma(\cdot)$ is the Gamma function. In addition, applying the same method as in the proof of (\ref{eq19}), we can conclude that
\begin{align}\label{eq20}
&\mathbb{E}\|U(\omega)\|_{C_\tau}^2\nonumber\\
\leq&2\|\Sigma\|^2\cdot\mathbb{E}\sup_{-\tau\leq s\leq0}|B(s,\omega)|^2\nonumber\\
&+2\mathbb{E}\left(\sup_{-\tau\leq s\leq0}\int_{-\infty}^{s}\left|\int_{-\tau}^0\mu(d\rho)r(s-u+\rho)\right|\cdot\|\Sigma\|\cdot|B(u,\omega)|du\right)^2\nonumber\\
\leq&2\|\Sigma\|^2\sum_{i=1}^m\left(\mathbb{E}\sup_{-\tau\leq s\leq0}|B_i(s,\omega)|^2\right)\nonumber\\
&+2e^{2\alpha\tau}C_\alpha^2\|\Sigma\|^2\left|\int_{-\tau}^0e^{-\alpha\rho}|\mu|(d\rho)\right|^2\mathbb{E}\left(\int_{-\infty}^0e^{\alpha u}|B(u,\omega)|du\right)^2\nonumber\\
\leq&8\|\Sigma\|^2\sum_{i=1}^m\mathbb{E}|B_i(-\tau,\omega)|^2\nonumber\\
&+2e^{2\alpha\tau}C_\alpha^2\|\Sigma\|^2\left|\int_{-\tau}^0e^{-\alpha\rho}|\mu|(d\rho)\right|^2\mathbb{E}\left(\int_{-\infty}^0e^{\alpha u}du\int_{-\infty}^0e^{\alpha u}|B(u,\omega)|^2du\right)\nonumber\\
=&8m\tau\|\Sigma\|^2+2me^{2\alpha\tau}C_\alpha^2\|\Sigma\|^2\left|\int_{-\tau}^0e^{-\alpha\rho}|\mu|(d\rho)\right|^2\cdot\frac1\alpha\int_{-\infty}^0-ue^{\alpha u}du\nonumber\\
=&8m\tau\|\Sigma\|^2+\frac{2m}{\alpha^3}e^{2\alpha\tau}C_\alpha^2\|\Sigma\|^2\left|\int_{-\tau}^0e^{-\alpha\rho}|\mu|(d\rho)\right|^2,
\end{align}
which proves (\ref{OU7}). The proof is complete.
\quad\quad$\Box$

\begin{remark}
{\rm In the proof of $\rm (ii)$, the continuity of $U(\theta_t\omega)$ for $t\geq0$ can also be obtained by  (\ref{Conclu3}). Furthermore, the so-called generalized Ornstein-Uhlenbeck process $U(\theta_t\omega)$ and its statistical properties will be useful on the conjugacy of stochastic and random functional differential equations, see Lemma \ref{lem2} below.}
\end{remark}

\noindent
{\bf Proof of Theorem \ref{thm1-1}.} Thanks to Theorem \ref{thm1} and Lebesgue's dominated convergence theorem, in order to prove (\ref{Conclu1-1}), it remains to show that there exists a random variable $F_\xi\in L^2(\Omega,\mathscr{F},\mathbb{P};\mathbb{R}_+)$ such that
\begin{equation}\label{eq20-1}
\|\varphi(t,\theta_{-t}\omega,\xi)-U(\omega)\|_{C_\tau}\leq F_\xi(\omega)
\end{equation}
for all $\xi\in C_\tau$, $t\geq\tau$ and $\omega\in\Omega$. Note that for $t\geq\tau$, by (\ref{eq2}), (\ref{eq3}) and (\ref{eq4}), we have
\begin{align}\label{eq20-2}
&\|\varphi(t,\theta_{-t}\omega,\xi)\|_{C_\tau}\nonumber\\
=&\sup_{-\tau\leq s\leq0}\left|\varphi(t,\theta_{-t}\omega,\xi)(s)\right|\nonumber\\
\leq&e^{\alpha\tau}C_\alpha e^{-\alpha t}\|\xi\|_{C_\tau}+\tau e^{2\alpha\tau}C_\alpha e^{-\alpha t}\|\xi\|_{C_\tau}\cdot|\mu|\bigl([-\tau,0]\bigr)\nonumber\\
&+\sup_{-\tau\leq s\leq0}\left|\int_{-t}^s r(s-u)\Sigma dB(u)\right|\nonumber\\
\leq&C_\alpha\|\xi\|_{C_\tau}+\tau e^{\alpha\tau}C_\alpha\|\xi\|_{C_\tau}\cdot|\mu|\bigl([-\tau,0]\bigr)\nonumber\\
&+\sup_{-\tau\leq s\leq0}\left|\int_{-t}^s r(s-u)\Sigma dB(u)\right|.
\end{align}
Combining this and the fact that $\|U(\bullet)\|_{C_\tau}\in L^2(\Omega,\mathscr{F},\mathbb{P};\mathbb{R}_+)$, see (v) of Proposition \ref{prop1}, we only need to verify that there exists $\widetilde{F}\in L^2(\Omega,\mathscr{F},\mathbb{P};\mathbb{R}_+)$ such that
\begin{equation}\label{eq20-3}
\sup_{-\tau\leq s\leq0}\left|\int_{-t}^s r(s-u)\Sigma dB(u)\right|\leq \widetilde{F}(\omega)
\end{equation}
for all $t\geq\tau$ and $\omega\in\Omega$. Applying the integration by parts formula, it is easy to check that
\begin{align}
&\int_{-t}^s r(s-u)\Sigma dB(u)\nonumber\\
=&\Sigma B(s)-r(t+s)\Sigma B(-t)-\int_{-t}^{s}\frac{\rm d}{{\rm d}u}\bigl[r(s-u)\bigr]\Sigma B(u)du\nonumber\\
=&\Sigma B(s)-r(t+s)\Sigma B(-t)+\int_{-t}^{s}\left[\int_{-\tau}^0\mu(d\rho)r(s-u+\rho)\right]\Sigma B(u)du\nonumber
\end{align}
and then
\begin{align}\label{eq20-4}
&\sup_{-\tau\leq s\leq0}\left|\int_{-t}^s r(s-u)\Sigma dB(u)\right|\nonumber\\
\leq&\|\Sigma\|\sup_{-\tau\leq s\leq0}|B(s)|+e^{\alpha\tau}C_\alpha\|\Sigma\|\sup_{t\geq\tau}\left\{e^{-\alpha t}|B(-t)|\right\} \nonumber\\
&+e^{2\alpha\tau}C_\alpha\|\Sigma\|\cdot|\mu|\bigl([-\tau,0]\bigr)\int_{-\infty}^{0}e^{\alpha u}|B(u)|du\nonumber\\
\leq&\|\Sigma\|\sup_{-\tau\leq s\leq0}|B(s)|+e^{\alpha\tau}C_\alpha\|\Sigma\|\sup_{t\geq0}\left\{e^{-\frac\alpha2t}|B(-t)|\right\}\nonumber\\
&+e^{2\alpha\tau}C_\alpha\|\Sigma\|\cdot|\mu|\bigl([-\tau,0]\bigr)\sup_{t\geq0}\left\{e^{-\frac\alpha2t}|B(-t)|\right\}\int_{-\infty}^{0}e^{\frac\alpha2u}du\nonumber\\
\triangleq&\widetilde{F}_1(\omega)+\widetilde{F}_2(\omega)+\widetilde{F}_3(\omega).
\end{align}
In what follows, we will show that $\widetilde{F}_i\in L^2(\Omega,\mathscr{F},\mathbb{P};\mathbb{R}_+)$ for $i=1,2,3$, it remains to prove that
\begin{equation}\label{eq20-5}
\sup_{-\tau\leq s\leq0}|B(s)|\in L^2(\Omega,\mathscr{F},\mathbb{P};\mathbb{R}_+)
\end{equation}
and
\begin{equation}\label{eq20-6}
\sup_{t\geq0}\left\{e^{-\frac\alpha2t}|B(-t)|\right\}\in L^2(\Omega,\mathscr{F},\mathbb{P};\mathbb{R}_+).
\end{equation}
In fact, by Doob's martingale inequality, we conclude at once that
\begin{equation}\label{eq20-7}
\mathbb{E}\left(\sup_{-\tau\leq s\leq0}|B(s)|^2\right)\leq\sum_{i=1}^m\mathbb{E}\left(\sup_{-\tau\leq s\leq0}|B_i(s)|^2\right)\leq4\sum_{i=1}^m\mathbb{E}|B_i(-\tau)|^2=4m\tau
\end{equation}
and
\begin{align}\label{eq20-8}
&\mathbb{E}\left|\sup_{t\geq0}\left\{e^{-\frac\alpha2t}|B(-t)|\right\}\right|^2\nonumber\\
=&\mathbb{E}\left(\sup_{t\geq0}\left\{e^{-\alpha t}|B(-t)|^2\right\}\right)\nonumber\\
\leq&\sum_{i=1}^m\mathbb{E}\left(\sup_{n\in\mathbb{N}}\sup_{n-1\leq t\leq n}\left\{e^{-\alpha t}|B_i(-t)|^2\right\}\right)\nonumber\\
\leq&\sum_{i=1}^m\sum_{n=1}^\infty\mathbb{E}\left(\sup_{n-1\leq t\leq n}\left\{e^{-\alpha t}|B_i(-t)|^2\right\}\right)\nonumber\\
\leq&\sum_{i=1}^m\sum_{n=1}^\infty e^{-\alpha(n-1)}\mathbb{E}\left(\sup_{n-1\leq t\leq n}\left\{|B_i(-t)|^2\right\}\right)\nonumber\\
\leq&4\sum_{i=1}^m\sum_{n=1}^\infty e^{-\alpha(n-1)}\mathbb{E}|B_i(-n)|^2\nonumber\\
=&4m\sum_{n=1}^\infty ne^{-\alpha(n-1)}\nonumber\\
<&\infty,
\end{align}
which gives (\ref{eq20-5}) and (\ref{eq20-6}). Set $\widetilde{F}=\widetilde{F}_1+\widetilde{F}_2+\widetilde{F}_3$, the proof is complete.\quad\quad$\Box$

\noindent
{\bf Proof of Corollary \ref{cor1}.} Based on Theorem \ref{thm1-1} and H\"{o}lder's inequality, the proof is immediate.\quad\quad$\Box$

\section{Nonlinear system and global stability}
In this section, we will focus on the nonlinear perturbation of (\ref{SFDE}) and (\ref{SFDE2}), i.e.
\begin{equation}\label{SFDE3}
  dx(t)=[L(x_t)+f(x_t)]dt+\Sigma dB(t),\quad t\geq0,
\end{equation}
\begin{equation}\label{SFDE4}x(s)=x_0(s)=\xi(s),\quad s\in[-\tau,0]\ {\rm and}\ \xi\in C_\tau,
\end{equation}
where $f:C_\tau\rightarrow\mathbb{R}^n$ satisfies the globally Lipschitz condition
\begin{equation}\label{Lip}
\bigl|f(\xi)-f(\eta)\bigr|\leq L\|\xi-\eta\|_{C_\tau}
\end{equation}
for all $\xi,\eta\in C_\tau$, where $L>0$ is the Lipschitz constant.  Let $\psi(t,\omega,\xi)=x_t(\omega,\xi)$ denote the solution mapping of (\ref{SFDE3}) and (\ref{SFDE4}), which can generate an RDS in the state space $C_\tau$, see Theorem 2.9 in \cite{BJX} or \cite{JL3}. In what follows, we will show that the asymptotic behavior of $\psi$ will be similar as $\varphi$ generated by (\ref{SFDE}) and (\ref{SFDE2}).

\begin{theorem}\label{thm2} Assume that $\alpha_0<0$ and there exists a constant $\alpha\in(0,-\alpha_0)$ such that $Le^{\alpha\tau}C_\alpha-\alpha<0$, then there exists a unique random equilibrium $V(\omega)$ of the RDS $(\theta,\psi)$ such that
\begin{equation}\label{Conclu4}
\lim_{t\rightarrow\infty}\psi(t,\theta_{-t}\omega,\xi)=V(\omega)\quad {\rm in}\quad C_\tau
\end{equation}
for any $\xi\in C_\tau$ and $\omega\in\Omega$. Moreover, the random equilibrium $V:\Omega\rightarrow C_\tau$ is tempered.
\end{theorem}

\begin{remark}
{\rm Compared to the results in \cite{JL1,JL3}, we remove some assumptions of boundedness and monotonicity. Moreover, the proofs given in the next section are greatly simplified.
}
\end{remark}

\begin{remark}
{\rm The proof of Theorem \ref{thm2} presented in Lemma \ref{lem1}-Lemma \ref{lem4} is standard and rigorous, which can also be applied to investigate the existence and global stability of random equilibria for the following nonlinear SFDEs
\begin{equation}\label{SFDE3-2}
  dx(t)=g(x_t)dt+\Sigma dB(t),\quad t\geq0.
\end{equation}
Here, $g:C_\tau\rightarrow\mathbb{R}^n$ satisfies the locally Lipschitz condition and the strong dissipative condition
\[\bigl(\xi(0)-\eta(0),g(\xi)-g(\eta)\bigr)\leq-\lambda_1|\xi(0)-\eta(0)|^2+\lambda_2\int_{-\tau}^0|\xi(s)-\eta(s)|^2\mu(ds),\]
where $\xi,\eta\in C_\tau$, $\lambda_1>\lambda_2>0$ and $\mu$ is a probability measure on $[-\tau,0]$.
}
\end{remark}

\begin{corollary}\label{cor2-1} Under the same conditions given in Theorem \ref{thm2}, then for all different initial values $\xi,\eta\in C_\tau$,
\begin{equation}\label{Conclu4-1}
\lim_{t\rightarrow\infty}\|\psi(t,\omega,\xi)-\psi(t,\omega,\eta)\|_{C_\tau}=0
\end{equation}
and
\begin{equation}\label{Conclu4-2}
\lim_{t\rightarrow\infty}|x(t,\omega,\xi)-x(t,\omega,\eta)|=0
\end{equation}
in probability. Here, $x(t,\omega,\xi)=\psi(t,\omega,\xi)[0]$ is the solution of (\ref{SFDE3}) and (\ref{SFDE4}).
\end{corollary}

\begin{theorem}\label{thm3} Under the same conditions given in Theorem \ref{thm2}, then
\begin{equation}\label{Conclu5}
\lim_{t\rightarrow\infty}\mathbb{E}\|\psi(t,\theta_{-t}\omega,\xi)-V(\omega)\|^2_{C_\tau}=0,
\end{equation}
where $V(\omega)$ is given in (\ref{Conclu4}).
\end{theorem}

\begin{corollary}\label{cor2} Under the same conditions given in Theorem \ref{thm2}, then
\begin{equation}\label{Conclu6}
\lim_{t\rightarrow\infty}\mathbb{E}\|\psi(t,\theta_{-t}\omega,\xi)-V(\omega)\|_{C_\tau}=0,
\end{equation}
where $V(\omega)$ is given in (\ref{Conclu4}).
\end{corollary}

\section{Proofs of main results in Section 4}
\noindent
\begin{lemma}\label{lem1}
Suppose that $\alpha_0<0$ and (\ref{Lip}) holds, then for any $\xi,\eta\in C_\tau$, we get hat
\begin{equation}\label{Conclu7}
\|\psi(t,\omega,\xi)-\psi(t,\omega,\eta)\|_{C_\tau}\leq K\|\xi-\eta\|_{C_\tau}e^{(Le^{\alpha\tau}C_\alpha-\alpha)t},\ t\geq\tau,\ \omega\in\Omega,
\end{equation}
where $K=e^{\alpha\tau}C_\alpha+\tau e^{2\alpha\tau}C_\alpha\cdot|\mu|\bigl([-\tau,0]\bigr)$.
\end{lemma}
\begin{proof}
By (\ref{SFDE3}), $-\tau\leq s\leq0$ and the variation-of-constants formula, see Lemma 6.1 in \cite{RRG}, for $t\geq\tau$, it is simple matter to check that
\begin{align}\label{eq21}
&\psi(t,\omega,\xi)(s)\nonumber\\
=&r(t+s)\xi(0)+\int_{-\tau}^0\int_{-\tau}^{u}r(t+s+\rho-u)\mu(d\rho)\xi(u)du\nonumber\\
&+\int_0^{t+s}r(t+s-\rho)f\bigl(\psi(\rho,\omega,\xi)\bigr)d\rho+\int_0^{t+s}r(t+s-\rho)\Sigma dB(\rho).
\end{align}
So, choose two different initial values $\xi,\eta\in C_\tau$, we see at once that
\begin{align}\label{eq22}
&|\psi(t,\omega,\xi)(s)-\psi(t,\omega,\eta)(s)|\nonumber\\
\leq&\|r(t+s)\|\cdot|\xi(0)-\eta(0)|+\int_{-\tau}^0\left\|\int_{-\tau}^{u}r(t+s+\rho-u)\mu(d\rho)\right\|\cdot|\xi(u)-\eta(u)|du\nonumber\\
&+\int_0^{t+s}\|r(t+s-\rho)\|\cdot|f\bigl(\psi(\rho,\omega,\xi)\bigr)-f\bigl(\psi(\rho,\omega,\eta)\bigr)|d\rho\nonumber\\
\leq&e^{\alpha\tau}C_\alpha e^{-\alpha t}\|\xi-\eta\|_{C_\tau}+\tau e^{2\alpha\tau}C_\alpha\cdot|\mu|\bigl([-\tau,0]\bigr)e^{-\alpha t}\|\xi-\eta\|_{C_\tau}\nonumber\\
&+Le^{\alpha\tau}C_\alpha e^{-\alpha t}\int_0^{t}e^{\alpha\rho}\|\psi(\rho,\omega,\xi)-\psi(\rho,\omega,\eta)\|_{C_\tau}d\rho,
\end{align}
which implies that
\begin{align}\label{eq23}
&e^{\alpha t}\|\psi(t,\omega,\xi)-\psi(t,\omega,\eta)\|_{C_\tau}\nonumber\\
\leq&K\|\xi-\eta\|_{C_\tau}+Le^{\alpha\tau}C_\alpha\int_0^{t}e^{\alpha\rho}\|\psi(\rho,\omega,\xi)-\psi(\rho,\omega,\eta)\|_{C_\tau}d\rho.
\end{align}
Here, we write $K=e^{\alpha\tau}C_\alpha+\tau e^{2\alpha\tau}C_\alpha\cdot|\mu|\bigl([-\tau,0]\bigr)$. Combining (\ref{eq23}) and the Gronwall inequality, it is obvious that
\begin{equation}\label{eq24}
\|\psi(t,\omega,\xi)-\psi(t,\omega,\eta)\|_{C_\tau}\leq K\|\xi-\eta\|_{C_\tau}e^{(Le^{\alpha\tau}C_\alpha-\alpha)t}
\end{equation}
for all $t\geq\tau$ and $\omega\in\Omega$. The proof is complete.
\end{proof}

\begin{remark}
{\rm If $Le^{\alpha\tau}C_\alpha-\alpha<0$, Lemma \ref{lem1} shows that the phenomenon of synchronization will occur, see \cite{CK,FGS1,FGS2}. In the following lemmas, we will give a rigorous proof of the existence and uniqueness of globally stable random equilibria for the RDS $(\theta,\psi)$. }
\end{remark}

\begin{lemma}\label{lem2}
Suppose that $\alpha_0<0$ and there exists a constant $\alpha\in(0,-\alpha_0)$ such that $Le^{\alpha\tau}C_\alpha-\alpha<0$, then for any $\xi\in C_\tau$,
\begin{equation}\label{Conclu8}
\sup_{t\geq0}\|\psi(t,\theta_{-t}\omega,\xi)\|_{C_\tau}\leq R_\xi(\omega)
\end{equation}
for all $\omega\in\Omega$, where $R_\xi(\omega)$ is a tempered random variable. Here, the random variable $R_\xi(\omega)$ depends on $\xi$ and $L$ is the Lipschitz constant given in (\ref{Lip}).
\end{lemma}
\begin{proof}
Let us first observe that for any $\xi\in C_\tau$, there exists a tempered random variable $R_\xi^1(\omega)$ such that
\begin{equation}\label{eq25}
\sup_{t\geq\tau}\|\psi(t,\theta_{-t}\omega,\xi)\|_{C_\tau}\leq R_\xi^1(\omega),\ \omega\in\Omega.
\end{equation}
By (\ref{Conclu3}), we define
\begin{equation}\label{eq26}\phi\bigl(t,\omega,\xi-U(\omega)\bigr)=\psi(t,\omega,\xi)-\varphi\bigl(t,\omega,U(\omega)\bigr)=\psi(t,\omega,\xi)-U(\theta_{t}\omega),
\end{equation}
which together with (\ref{OU3-2}), (\ref{SFDE3}) and It\^{o}'s formula gives that
\begin{equation}\label{eq27}
\frac{\rm d}{{\rm d}t}\Bigl[\phi\bigl(t,\omega,\phi_0(\omega)\bigr)(0)\Bigr]
=L\Bigl[\phi\bigl(t,\omega,\phi_0(\omega)\bigr)\Bigr]+f\Bigl(\phi\bigl(t,\omega,\phi_0(\omega)\bigr)+U(\theta_{t}\omega)\Bigr),
\end{equation}
where $\phi_0(\omega)=\xi-U(\omega)$.

For $t\geq\tau$ and $-\tau\leq s\leq0$, using the variation-of-constants formula, see Theorem 1.2 in \cite[Chapter 6]{HL}, it follows immediately that
\begin{align}\label{eq28}
&\phi\bigl(t,\omega,\phi_0(\omega)\bigr)(s)\nonumber\\
=&r(t+s)\phi_0(\omega)(0)+\int_{-\tau}^0\int_{-\tau}^{u}r(t+s+\rho-u)\mu(d\rho)\phi_0(\omega)(u)du\nonumber\\
&+\int_0^{t+s}r(t+s-\rho)f\Bigl(\phi\bigl(\rho,\omega,\phi_0(\omega)\bigr)+U(\theta_{\rho}\omega)\Bigr)d\rho
\end{align}
and then
\begin{align}\label{eq29}
&\left\|\phi\bigl(t,\omega,\phi_0(\omega)\bigr)\right\|_{C_\tau}\nonumber\\
=&\sup_{-\tau\leq s\leq0}\left|\phi\bigl(t,\omega,\phi_0(\omega)\bigr)(s)\right|\nonumber\\
\leq&e^{\alpha\tau}C_\alpha e^{-\alpha t}\|\phi_0(\omega)\|_{C_\tau}+\tau e^{2\alpha\tau}C_\alpha e^{-\alpha t}|\mu|\bigl([-\tau,0]\bigr)\cdot\|\phi_0(\omega)\|_{C_\tau}\nonumber\\
&+e^{\alpha\tau}C_\alpha e^{-\alpha t}\int_0^{t}e^{\alpha\rho}\bigl(L\|U(\theta_{\rho}\omega)\|_{C_\tau}+|f(0)|\bigr)d\rho\nonumber\\
&+Le^{\alpha\tau}C_\alpha e^{-\alpha t}\int_0^{t}e^{\alpha\rho}\left\|\phi\bigl(\rho,\omega,\phi_0(\omega)\bigr)\right\|_{C_\tau}d\rho\nonumber\\
=&Ke^{-\alpha t}\|\phi_0(\omega)\|_{C_\tau}+e^{\alpha\tau}C_\alpha e^{-\alpha t}\int_0^{t}e^{\alpha\rho}\bigl(L\|U(\theta_{\rho}\omega)\|_{C_\tau}+|f(0)|\bigr)d\rho\nonumber\\
&+Le^{\alpha\tau}C_\alpha e^{-\alpha t}\int_0^{t}e^{\alpha\rho}\left\|\phi\bigl(\rho,\omega,\phi_0(\omega)\bigr)\right\|_{C_\tau}d\rho,
\end{align}
where
\[K=K(\alpha,\tau,\mu)=e^{\alpha\tau}C_\alpha+\tau e^{2\alpha\tau}C_\alpha|\mu|\bigl([-\tau,0]\bigr).\]
Therefore, it is obvious that
\begin{align}\label{eq30}
&e^{\alpha t}\left\|\phi\bigl(t,\omega,\phi_0(\omega)\bigr)\right\|_{C_\tau}\nonumber\\
\leq&K\|\phi_0(\omega)\|_{C_\tau}+e^{\alpha\tau}C_\alpha\int_0^{t}e^{\alpha\rho}\bigl(L\|U(\theta_{\rho}\omega)\|_{C_\tau}+|f(0)|\bigr)d\rho\nonumber\\
&+Le^{\alpha\tau}C_\alpha\int_0^{t}e^{\alpha\rho}\left\|\phi\bigl(\rho,\omega,\phi_0(\omega)\bigr)\right\|_{C_\tau}d\rho.
\end{align}
Combining this and the Gronwall inequality, we can easily obtain that
\begin{align}\label{eq31}
&e^{\alpha t}\left\|\phi\bigl(t,\omega,\phi_0(\omega)\bigr)\right\|_{C_\tau}\nonumber\\
\leq&K\|\phi_0(\omega)\|_{C_\tau}e^{Le^{\alpha\tau}C_\alpha t}+e^{\alpha\tau}C_\alpha\int_0^{t}e^{\alpha\rho}e^{Le^{\alpha\tau}C_\alpha (t-\rho)}\bigl(L\|U(\theta_{\rho}\omega)\|_{C_\tau}+|f(0)|\bigr)d\rho
\end{align}
and then
\begin{align}\label{eq32}
&\left\|\phi\bigl(t,\omega,\phi_0(\omega)\bigr)\right\|_{C_\tau}\nonumber\\
\leq&K\|\phi_0(\omega)\|_{C_\tau}e^{-(\alpha-Le^{\alpha\tau}C_\alpha)t}\nonumber\\
&+e^{\alpha\tau}C_\alpha\int_0^{t}e^{-(\alpha-Le^{\alpha\tau}C_\alpha) (t-\rho)}\bigl(L\|U(\theta_{\rho}\omega)\|_{C_\tau}+|f(0)|\bigr)d\rho,\ t\geq\tau.
\end{align}
This together with definitions of $\theta$ and $\phi_0$ shows that
\begin{align}\label{eq32-1}
&\left\|\phi\bigl(t,\theta_{-t}\omega,\phi_0(\theta_{-t}\omega)\bigr)\right\|_{C_\tau}\nonumber\\
\leq&K\|\phi_0(\theta_{-t}\omega)\|_{C_\tau}e^{-(\alpha-Le^{\alpha\tau}C_\alpha)t}\nonumber\\
&+e^{\alpha\tau}C_\alpha\int_0^{t}e^{-(\alpha-Le^{\alpha\tau}C_\alpha) (t-\rho)}\bigl(L\|U(\theta_{-(t-\rho)}\omega)\|_{C_\tau}+|f(0)|\bigr)d\rho\nonumber\\
\leq&K\|\xi\|_{C_\tau}e^{-(\alpha-Le^{\alpha\tau}C_\alpha)t}+K\|U(\theta_{-t}\omega)\|_{C_\tau}e^{-(\alpha-Le^{\alpha\tau}C_\alpha)t}\nonumber\\
&+e^{\alpha\tau}C_\alpha\int_0^{t}e^{-(\alpha-Le^{\alpha\tau}C_\alpha) (t-\rho)}\bigl(L\|U(\theta_{-(t-\rho)}\omega)\|_{C_\tau}+|f(0)|\bigr)d\rho\nonumber\\
=&K\|\xi\|_{C_\tau}e^{-(\alpha-Le^{\alpha\tau}C_\alpha)t}+Ke^{-(\alpha-Le^{\alpha\tau}C_\alpha)t}\|U(\theta_{-t}\omega)\|_{C_\tau}\nonumber\\
&+e^{\alpha\tau}C_\alpha\int_{-t}^0e^{(\alpha-Le^{\alpha\tau}C_\alpha)\rho}\bigl(L\|U(\theta_{\rho}\omega)\|_{C_\tau}+|f(0)|\bigr)d\rho,\ t\geq\tau.
\end{align}
Since $Le^{\alpha\tau}C_\alpha-\alpha<0$ and $U(\omega)$ is tempered (see Proposition \ref{prop1}), we set
\begin{align}\label{eq33}
\widetilde{R}_\xi(\omega)=&K\|\xi\|_{C_\tau}+K\sup_{t\geq0}\left\{e^{-(\alpha-Le^{\alpha\tau}C_\alpha)t}\|U(\theta_{-t}\omega)\|_{C_\tau}\right\}\nonumber\\
&+e^{\alpha\tau}C_\alpha\int_{-\infty}^0e^{(\alpha-Le^{\alpha\tau}C_\alpha)\rho}\bigl(L\|U(\theta_{\rho}\omega)\|_{C_\tau}+|f(0)|\bigr)d\rho
\end{align}
and then
\begin{equation}\label{eq33-1}
\sup_{t\geq\tau}\left\|\phi\bigl(t,\theta_{-t}\omega,\phi_0(\theta_{-t}\omega)\bigr)\right\|_{C_\tau}\leq\widetilde{R}_\xi(\omega),
\end{equation}
where we can see that the random variable $\widetilde{R}_\xi(\omega)$ is well defined for all $\omega\in\Omega$. In what follows, we will check that $\widetilde{R}_\xi(\omega)$ is tempered. To this end, set $\delta=\alpha-Le^{\alpha\tau}C_\alpha>0$, we only need to prove that for all $\omega\in\Omega$ and $\gamma>0$
\begin{equation}\label{eq34}
\sup_{t\in\mathbb{R}}
\left\{e^{-\gamma|t|}\sup_{u\geq0}\left\{e^{-\delta u}\|U(\theta_{-u}\circ\theta_t\omega)\|_{C_\tau}\right\}\right\}<\infty\end{equation}
and
\begin{equation}\label{eq35}\sup_{t\in\mathbb{R}}
\left\{e^{-\gamma|t|}\int_{-\infty}^0e^{\delta\rho}\|U(\theta_{\rho}\circ\theta_t\omega)\|_{C_\tau}d\rho\right\}<\infty.\end{equation}
In fact, for any $\gamma>0$,
\begin{align}\label{eq36}
&\sup_{t\in\mathbb{R}}\left\{e^{-\gamma|t|}\sup_{u\geq0}\left\{e^{-\delta u}\|U(\theta_{-u}\circ\theta_t\omega)\|_{C_\tau}\right\}\right\}\nonumber\\
\leq&\sup_{t\in\mathbb{R}}\left\{e^{-\gamma|t|}\sup_{u\geq0}\left\{e^{-(\delta\wedge\gamma)u}\|U(\theta_{-u}\circ\theta_t\omega)\|_{C_\tau}\right\}\right\}\nonumber\\
=&\sup_{t\in\mathbb{R}}\left\{e^{-\gamma|t|}e^{(\delta\wedge\gamma)|t|}\sup_{u\geq0}\left\{e^{-(\delta\wedge\gamma)u}e^{-(\delta\wedge\gamma)|t|}\|U(\theta_{t-u}\omega)\|_{C_\tau}\right\}\right\}\nonumber\\
\leq&\sup_{t\in\mathbb{R}}\left\{e^{-\gamma|t|}e^{(\delta\wedge\gamma)|t|}\sup_{u\geq0}\left\{e^{-(\delta\wedge\gamma)|t-u|}\|U(\theta_{t-u}\omega)\|_{C_\tau}\right\}\right\}\nonumber\\
\leq&\sup_{u\in\mathbb{R}}\left\{e^{-(\delta\wedge\gamma)|u|}\|U(\theta_{u}\omega)\|_{C_\tau}\right\}\nonumber\\
<&\infty,
\end{align}
which is based on (\ref{OU2}). On the other hand, for any $\gamma>0$,
\begin{align}\label{eq37}
&\sup_{t\in\mathbb{R}}\left\{e^{-\gamma|t|}\int_{-\infty}^0e^{\delta\rho}\|U(\theta_{\rho}\circ\theta_t\omega)\|_{C_\tau}d\rho\right\}\nonumber\\
\leq&\sup_{t\in\mathbb{R}}\left\{e^{-\gamma|t|}\int_{-\infty}^0 e^{-\frac\delta2|\rho|}e^{-(\frac\delta2\wedge\gamma)|\rho|}\|U(\theta_{\rho+t}\omega)\|_{C_\tau}d\rho\right\}\nonumber\\
\leq&\sup_{t\in\mathbb{R}}\left\{e^{-\gamma|t|}e^{(\frac\delta2\wedge\gamma)|t|}\int_{-\infty}^0 e^{-\frac\delta2|\rho|}e^{-(\frac\delta2\wedge\gamma)|\rho+t|}\|U(\theta_{\rho+t}\omega)\|_{C_\tau}d\rho\right\}\nonumber\\
\leq&\sup_{u\in\mathbb{R}}\left\{e^{-(\frac\delta2\wedge\gamma)|u|}\|U(\theta_{u}\omega)\|_{C_\tau}\right\}\int_{-\infty}^0 e^{-\frac\delta2|\rho|}d\rho\nonumber\\
<&\infty.
\end{align}
That is, the random variable $\widetilde{R}_\xi(\omega)$ is tempered.  Thus, from (\ref{eq26}) and (\ref{eq33-1}), we have immediately that
\begin{align}\label{eq38}
&\sup_{t\geq\tau}\|\psi(t,\theta_{-t}\omega,\xi)\|_{C_\tau}\nonumber\\
\leq&\sup_{t\geq\tau}\left\|\phi\bigl(t,\theta_{-t}\omega,\phi_0(\theta_{-t}\omega)\bigr)\right\|_{C_\tau}+\left\|U(\omega)\right\|_{C_\tau}\nonumber\\
\leq&\widetilde{R}_\xi(\omega)+\left\|U(\omega)\right\|_{C_\tau}\nonumber\\
\triangleq&R_\xi^1(\omega),\quad\omega\in\Omega,
\end{align}
which proves (\ref{eq25}). Our next claim is that there exists another tempered random variable $R_\xi^2(\omega)$ such that
\begin{equation}\label{eq39}
\sup_{0\leq t\leq\tau}\|\psi(t,\theta_{-t}\omega,\xi)\|_{C_\tau}\leq R_\xi^2(\omega),\quad\omega\in\Omega.
\end{equation}
Note that for $0\leq t\leq\tau$ and $-\tau\leq s\leq0$,
\begin{equation}\label{eq40}
\psi(t,\omega,\xi)(s)=\xi(t+s),\quad\omega\in\Omega,
\end{equation}
which implies that
\begin{align}\label{eq41-1}
&\|\psi(t,\theta_{-t}\omega,\xi)\|_{C_\tau}\nonumber\\
=&\sup_{-\tau\leq s\leq0}|\psi(t,\theta_{-t}\omega,\xi)(s)|\nonumber\\
\leq&\sup_{-\tau\leq s\leq-t}|\xi(t+s)|+\sup_{-t\leq s\leq0}|\psi(t,\theta_{-t}\omega,\xi)(s)|\nonumber\\
\leq&\|\xi\|_{C_\tau}+\|\psi(\tau,\theta_{-t}\omega,\xi)\|_{C_\tau},\quad\omega\in\Omega.
\end{align}
Consequently,
\begin{align}\label{eq41}
&\sup_{0\leq t\leq\tau}\|\psi(t,\theta_{-t}\omega,\xi)\|_{C_\tau}\nonumber\\
\leq&\|\xi\|_{C_\tau}+\sup_{0\leq t\leq\tau}\|\psi(\tau,\theta_{-t}\omega,\xi)\|_{C_\tau}\nonumber\\
=&\|\xi\|_{C_\tau}+\sup_{0\leq t\leq\tau}\|\psi(\tau,\theta_{-\tau}\circ\theta_{\tau-t}\omega,\xi)\|_{C_\tau}\nonumber\\
\leq&\|\xi\|_{C_\tau}+\sup_{0\leq t\leq\tau}R_\xi^1(\theta_{\tau-t}\omega)\nonumber\\
=&\|\xi\|_{C_\tau}+\sup_{t\in[0,\tau]\cap\mathbb{Q}}R_\xi^1(\theta_{\tau-t}\omega)\nonumber\\
\triangleq&R_\xi^2(\omega),\quad\omega\in\Omega.
\end{align}
Here, we observe the fact that the mapping $t\rightarrow U(\theta_t\omega)$ is continuous, which together with (\ref{eq33}) and (\ref{eq38}) yields the continuity of $R_\xi^1(\theta_t\omega)$ with respect to $t$ and the measurability of $R_\xi^2(\omega)$.

Actually, $R_\xi^2(\omega)$ is also tempered. For any $\gamma>0$ and $\omega\in\Omega$, using (\ref{eq38}) and (\ref{eq41}), it is evident that
\begin{align}\label{eq42}
&\sup_{t\in\mathbb{R}}\left\{e^{-\gamma|t|}\sup_{u\in[0,\tau]\cap\mathbb{Q}}R_\xi^1(\theta_{\tau-u}\circ\theta_t\omega)\right\}\nonumber\\
\leq&\sup_{t\in\mathbb{R}}\left\{e^{-\gamma|t|}\sup_{u\in[0,\tau]\cap\mathbb{Q}}\widetilde{R}_\xi(\theta_{t+\tau-u}\omega)\right\}
+\sup_{t\in\mathbb{R}}\left\{e^{-\gamma|t|}\sup_{u\in[0,\tau]\cap\mathbb{Q}}\|U(\theta_{t+\tau-u}\omega)\|_{C_\tau}\right\}\nonumber\\
\leq&\sup_{t\in\mathbb{R}}\left\{e^{-\gamma|t|}\sup_{u\in[0,\tau]\cap\mathbb{Q}}e^{\gamma|t+\tau-u|}\right\}
\sup_{t\in\mathbb{R}}\left\{e^{-\gamma|t|}\widetilde{R}_\xi(\theta_{t}\omega)\right\}\nonumber\\
&+\sup_{t\in\mathbb{R}}\left\{e^{-\gamma|t|}\sup_{u\in[0,\tau]\cap\mathbb{Q}}e^{\gamma|t+\tau-u|}\right\}
\sup_{t\in\mathbb{R}}\left\{e^{-\gamma|t|}\|U(\theta_{t}\omega)\|_{C_\tau}\right\}\nonumber\\
\leq&e^{\gamma\tau}\left(\sup_{t\in\mathbb{R}}\left\{e^{-\gamma|t|}\widetilde{R}_\xi(\theta_{t}\omega)\right\}
+\sup_{t\in\mathbb{R}}\left\{e^{-\gamma|t|}\|U(\theta_{t}\omega)\|_{C_\tau}\right\}\right)\nonumber\\
<&\infty.
\end{align}

Finally, let $R_\xi(\omega)=\max\{R_\xi^1(\omega),R_\xi^2(\omega)\}$ for all $\omega\in\Omega$, which together with (\ref{eq25}) and (\ref{eq41}) ends the proof of Lemma \ref{lem2}.
\end{proof}

\begin{lemma}\label{lem3}
Under the same assumptions presented in Lemma \ref{lem2}, for any $\xi\in C_\tau$ and $\omega\in\Omega$, there exists an $\mathscr{F}_-$-measurable random variable $V_\xi:\Omega\rightarrow C_\tau$ such that
\begin{equation}\label{eq43}
\lim_{t\rightarrow\infty}\psi(t,\theta_{-t}\omega,\xi)(\bullet)=V_\xi(\omega)(\bullet)
\end{equation}
in $C_\tau$. Moreover, the random variable $V_\xi$ is a tempered random equilibrium of the RDS $(\theta,\psi)$.
\end{lemma}
\begin{proof}
Choose any $\tau\leq t_1<t_2$, using the cocycle property, Lemma \ref{lem1} and Lemma \ref{lem2}, we can easily conclude that
\begin{align}\label{eq44}
&\|\psi(t_2,\theta_{-t_2}\omega,\xi)-\psi(t_1,\theta_{-t_1}\omega,\xi)\|_{C_\tau}\nonumber\\
=&\|\psi\bigl(t_1,\theta_{-t_1}\omega,\psi(t_2-t_1,\theta_{-t_2}\omega,\xi)\bigr)-\psi(t_1,\theta_{-t_1}\omega,\xi)\|_{C_\tau}\nonumber\\
\leq&K e^{-(\alpha-Le^{\alpha\tau}C_\alpha)t_1}\|\psi(t_2-t_1,\theta_{-(t_2-t_1)}\circ\theta_{-t_1}\omega,\xi)-\xi\|_{C_\tau}\nonumber\\
\leq&K e^{-(\alpha-Le^{\alpha\tau}C_\alpha)t_1}\left(\sup_{t\geq0}\|\psi(t,\theta_{-t}\circ\theta_{-t_1}\omega,\xi)\|_{C_\tau}+\|\xi\|_{C_\tau}\right)\nonumber\\
\leq&K e^{-(\alpha-Le^{\alpha\tau}C_\alpha)t_1}\bigl(R_\xi(\theta_{-t_1}\omega)+\|\xi\|_{C_\tau}\bigr)\nonumber\\
\leq&K e^{-\frac{\alpha-Le^{\alpha\tau}C_\alpha}{2}t_1}\left(\sup_{t\in\mathbb{R}}\left\{e^{-\frac{\alpha-Le^{\alpha\tau}C_\alpha}{2}|t|}R_\xi(\theta_{t}\omega)\right\}+\|\xi\|_{C_\tau}\right)\nonumber\\
&\longrightarrow0\quad {\rm as}\quad t_1\rightarrow\infty.
\end{align}
That is, for fixed $\xi\in C_\tau$ and $\omega\in\Omega$, $\{\psi(t,\theta_{-t}\omega,\xi):t\geq0\}$ is Cauchy in the Banach space $C_\tau$, which gives  (\ref{eq43}).

In addition, due to the fact that for all $-\tau\leq s\leq0$, $\psi(t,\omega,\xi)(s)$ is $\mathscr{F}_0^t$-adapted, where $\mathscr{F}_0^t=\sigma\{B(u):0\leq u\leq t\}$. This together with Lemma II.2.1 in \cite{Mo} shows that $\psi(t,\bullet,\xi):\Omega\rightarrow C_\tau$ is $\mathscr{F}_0^t$-measurable. Note that \begin{equation}\label{eq45}\theta_{-t}^{-1}\mathscr{F}_0^t=\mathscr{F}_{-t}^0=\sigma\{B(u):-t\leq u\leq 0\}\subset\mathscr{F}_-,\end{equation}
which implies that $\psi(t,\theta_{-t}\omega,\xi)$ is $\mathscr{F}_-$-measurable for all $t\geq0$. Consequently, the random variable $V_\xi:\Omega\rightarrow C_\tau$ is also $\mathscr{F}_-$-measurable.

Finally, using the similar argument in the proof of (\ref{eq10-2}), it is clear that $V_\xi$ is a random equilibrium of $(\theta,\psi)$. Moreover, by (\ref{eq38}) and (\ref{eq43}), we can deduce easily that
\begin{equation}\label{eq46}
\|V_\xi(\omega)\|_{C_\tau}\leq R_\xi^1(\omega),\quad\omega\in\Omega,
\end{equation}
where the random variable $R_\xi^1(\omega)$ is tempered given in (\ref{eq38}). The proof is complete.
\end{proof}

\begin{lemma}\label{lem4}
Under the same assumptions presented in Lemma \ref{lem2}, for any two tempered random equilibria $V_1(\omega)$ and $V_2(\omega)$, then
\begin{equation}\label{eq47}
V_1(\omega)=V_2(\omega),\quad\omega\in\Omega.
\end{equation}
\end{lemma}
\begin{proof}
Since $V_1(\omega)$ and $V_2(\omega)$ are two tempered random equilibria of $(\theta,\psi)$, it follows that
\[\psi\bigl(t,\theta_{-t}\omega,V_i(\theta_{-t}\omega)\bigr)=V_i(\omega),\quad i=1,2,\ t\geq0,\]
and there exists a tempered random variable $R:\Omega\rightarrow\mathbb{R}_+$ such that
\[\|V_1(\omega)-V_2(\omega)\|_{C_\tau}\leq R(\omega)\]
for all $\omega\in\Omega$. By (\ref{Conclu7}), it is obvious that
\begin{align}\label{eq48}
&\|V_1(\omega)-V_2(\omega)\|_{C_\tau}\nonumber\\
=&\|\psi\bigl(t,\theta_{-t}\omega,V_1(\theta_{-t}\omega)\bigr)-\psi\bigl(t,\theta_{-t}\omega,V_2(\theta_{-t}\omega)\bigr)\|_{C_\tau}\nonumber\\
\leq& K\|V_1(\theta_{-t}\omega)-V_2(\theta_{-t}\omega)\|_{C_\tau}e^{-(\alpha-Le^{\alpha\tau}C_\alpha)t}\nonumber\\
\leq& K\sup_{u\in\mathbb{R}}\left\{e^{-\frac{\delta}{2}|u|}R(\theta_u\omega)\right\}e^{-\frac\delta2t},
\end{align}
where $\delta=\alpha-Le^{\alpha\tau}C_\alpha$, $t\geq\tau$ and $\omega\in\Omega$. Let $t\rightarrow\infty$ in (\ref{eq48}), the proof is complete.
\end{proof}

\noindent
{\bf Proof of Theorem \ref{thm2}.} From Lemma \ref{lem3} and Lemma \ref{lem4}, for any $\xi,\eta\in C_\tau$, $V_\xi(\omega)=V_\eta(\omega)$ for all $\omega\in\Omega$. That is, the random equilibrium is unique, the proof is complete.\quad\quad$\Box$

\noindent
{\bf Proof of Corollary \ref{cor2-1}.} By Theorem \ref{thm2}, it is immediate that \begin{equation}\label{eq48-1}\lim_{t\rightarrow\infty}\|\psi(t,\theta_{-t}\omega,\xi)-\psi(t,\theta_{-t}\omega,\eta)\|_{C_\tau}=0
\end{equation}
for all $\omega\in\Omega$ and $\xi,\eta\in C_\tau$. Observe that $\theta_t\mathbb{P}=\mathbb{P}$ for all $t\in\mathbb{R}$, (\ref{eq48-1}) gives that (\ref{Conclu4-1}) and (\ref{Conclu4-2}) hold in probability. The proof is complete.\quad\quad$\Box$

\begin{lemma}\label{lem5}
For any $\gamma>0$, then
\begin{equation}\label{eq49}
\mathbb{E}\left|\sup_{t\geq0}\left\{e^{-\gamma t}\|U(\theta_{-t}\omega)\|_{C_\tau}\right\}\right|^2<\infty.
\end{equation}
\end{lemma}
\begin{proof}From (\ref{eq13}), it is evident that for $t\geq0$,
\begin{align}\label{eq50}
&e^{-\gamma t}\|U(\theta_{-t}\omega)\|_{C_\tau}\nonumber\\
\leq&\|\Sigma\|\cdot e^{-\gamma t}\sup_{-\tau\leq s\leq0}|B(s-t,\omega)|+\|\Sigma\|\cdot e^{-\gamma t}|B(-t,\omega)|\nonumber\\
&+\sup_{-\tau\leq s\leq0}\int_{-\infty}^{s}\left\|\int_{-\tau}^0\mu(d\rho)r(s-u+\rho)\right\|\cdot\|\Sigma\|\cdot e^{-\gamma t}|B(u-t,\omega)|du\nonumber\\
&+\left(\sup_{-\tau\leq s\leq0}\int_{-\infty}^{s}\left\|\int_{-\tau}^0\mu(d\rho)r(s-u+\rho)\right\|\cdot\|\Sigma\|du\right)e^{-\gamma t}|B(-t,\omega)|\nonumber\\
\leq&\|\Sigma\|e^{\gamma\tau}\sup_{t\geq0}\left\{e^{-\gamma t}|B(-t,\omega)|\right\}+\|\Sigma\|\sup_{t\geq0}\left\{e^{-\gamma t}|B(-t,\omega)|\right\}\nonumber\\
&+e^{2\alpha\tau}C_\alpha\|\Sigma\|\cdot|\mu|\bigl([-\tau,0]\bigr)\int_{-\infty}^{0}e^{\alpha u}e^{-\gamma t}|B(u-t,\omega)|du\nonumber\\
&+e^{2\alpha\tau}C_\alpha\|\Sigma\|\cdot|\mu|\bigl([-\tau,0]\bigr)\int_{-\infty}^{0}e^{\alpha u}du\cdot e^{-\gamma t}|B(-t,\omega)|\nonumber\\
\leq&\|\Sigma\|e^{\gamma\tau}\sup_{t\geq0}\left\{e^{-\gamma t}|B(-t,\omega)|\right\}+\|\Sigma\|\sup_{t\geq0}\left\{e^{-\gamma t}|B(-t,\omega)|\right\}\nonumber\\
&+e^{2\alpha\tau}C_\alpha\|\Sigma\|\cdot|\mu|\bigl([-\tau,0]\bigr)\int_{-\infty}^{0}e^{(\alpha-\frac{\gamma\wedge\alpha}{2})u}du\cdot\sup_{t\geq0}\left\{e^{-\frac{\gamma\wedge\alpha}{2} t}|B(-t,\omega)|\right\}\nonumber\\
&+e^{2\alpha\tau}C_\alpha\|\Sigma\|\cdot|\mu|\bigl([-\tau,0]\bigr)\int_{-\infty}^{0}e^{\alpha u}du\cdot\sup_{t\geq0}\left\{e^{-\gamma t}|B(-t,\omega)|\right\},
\end{align}
which implies that
\begin{align}\label{eq51}
&\sup_{t\geq0}\left\{e^{-\gamma t}\|U(\theta_{-t}\omega)\|_{C_\tau}\right\}\nonumber\\
\leq&K_1\sup_{t\geq0}\left\{e^{-\gamma t}|B(-t,\omega)|\right\}+K_2\sup_{t\geq0}\left\{e^{-\frac{\gamma\wedge\alpha}{2} t}|B(-t,\omega)|\right\}.
\end{align}
Here, $K_1$ and $K_2$ are two positive constants. Similar to the proof in (\ref{eq20-8}), the proof of (\ref{eq49}) is straightforward.
\end{proof}

\noindent
{\bf Proof of Theorem \ref{thm3}.} By  (\ref{eq38}) and the same method as in Theorem \ref{thm1-1}, we are reduced to prove that $R_\xi^1\in L^2(\Omega,\mathscr{F},\mathbb{P};\mathbb{R}_+)$, where $R_\xi^1(\omega)=\widetilde{R}_\xi(\omega)+\left\|U(\omega)\right\|_{C_\tau}$. Therefore, we are left with the task of showing that $\widetilde{R}_\xi\in L^2(\Omega,\mathscr{F},\mathbb{P};\mathbb{R}_+)$. Note that by (\ref{eq33}),
\begin{align}\label{eq52}
\widetilde{R}_\xi(\omega)=&K\|\xi\|_{C_\tau}+K\sup_{t\geq0}\left\{e^{-\delta t}\|U(\theta_{-t}\omega)\|_{C_\tau}\right\}\nonumber\\
&+e^{\alpha\tau}C_\alpha\int_{-\infty}^0e^{\delta\rho}\bigl(L\|U(\theta_{\rho}\omega)\|_{C_\tau}+|f(0)|\bigr)d\rho\nonumber\\
\leq&K\|\xi\|_{C_\tau}+K\sup_{t\geq0}\left\{e^{-\delta t}\|U(\theta_{-t}\omega)\|_{C_\tau}\right\}\nonumber\\
&+e^{\alpha\tau}C_\alpha\int_{-\infty}^0e^{\frac\delta2\rho}d\rho\left(L\sup_{t\geq0}\left\{e^{-\frac\delta2t}\|U(\theta_{-t}\omega)\|_{C_\tau}\right\}+|f(0)|\right),\nonumber
\end{align}
where $\delta=\alpha-Le^{\alpha\tau}C_\alpha>0$. Using Lemma \ref{lem5}, it follows obviously that $\mathbb{E}|\widetilde{R}_\xi|^2<\infty$. The proof is complete.
\quad\quad$\Box$

\noindent
{\bf Proof of Corollary \ref{cor2}.} Combining Theorem \ref{thm3} and H\"{o}lder's inequality, the proof is straightforward.\quad\quad$\Box$

\section{Examples}
In this section, we will give some examples to illustrate our main results. For simplicity, we assume that $n=m=1$.

\noindent
{\bf Example 6.1.} First, we consider the following $1$-dimensional affine stochastic delay differential equation (SDDE)
\begin{equation}\label{Ex6.1.1}
dx(t)=\bigl[-ax(t)+bx(t-1)\bigr]dt+\sigma dW_t,\end{equation}
with the initial value
\begin{equation}\label{Ex6.1.2}x(s)=x_0(s)=\xi(s),\quad s\in[-\tau,0]\ {\rm and}\ \xi\in C_\tau,
\end{equation}
where $a>0$, $b\in\mathbb{R}$, $\sigma>0$, $\tau=1$ and $W_t$ is a $1$-dimensional Brownian motion. If $|b|<a$, we can conclude that
\begin{equation}\label{Ex6.1.3}
|r(t)|\leq e^{(-a+\mu)t},\quad t\geq-1,\end{equation}
where $|b|<\mu<a$ and $\mu e^{-a+\mu}-|b|=0$. That is, we can choose $\alpha=a-\mu>0$ and $C_\alpha=1$ in (\ref{Eigen2}), see \cite[p.392]{LM}. Consequently, by Theorem \ref{thm1}, there exists a globally attracting stationary solution (generalized Ornstein-Uhlenbeck process) for (\ref{Ex6.1.1}) and (\ref{Ex6.1.2}) in the space $C_\tau$.

\noindent
{\bf Example 6.2.} Secondly, we consider the following $1$-dimensional nonlinear SDDE
\begin{equation}\label{Ex6.2.1}
dx(t)=\Bigl[-2x(t)+x(t-1)+\frac14\sin\bigl(x(t-1)\bigr)\Bigr]dt+\sigma dW_t,\end{equation}
with the initial value
\begin{equation}\label{Ex6.2.2}x(s)=x_0(s)=\xi(s),\quad s\in[-\tau,0]\ {\rm and}\ \xi\in C_\tau,
\end{equation}
where $\sigma>0$, $\tau=1$ and $W_t$ is a $1$-dimensional Brownian motion. Note that the unique solution of $\mu e^{-2+\mu}-1=0$ is $\mu\approx1.55786$, which together with (\ref{Ex6.1.3}) implies that
\begin{equation}\label{Ex6.2.3}
|r(t)|\leq e^{(-2+\mu)t}\leq e^{-0.4t},\quad t\geq-1.\end{equation}
Let $\alpha=0.4$, $C_\alpha=1$, $L=\frac14$, it follows that
\[Le^{\alpha\tau}C_\alpha-\alpha=\frac{e^{0.4}}{4}-0.4\approx-0.02704<0.\]
Therefore, applying Theorem \ref{thm2}, there is a unique globally stable stationary
solution for (\ref{Ex6.2.1}) and (\ref{Ex6.2.2}), which attracts all
pullback trajectories in the space $C_\tau$. Moreover, the synchronization of forward trajectories
will occur based on Corollary \ref{cor2-1}, see Fig. 1.

\begin{figure}[h]
\centering
\includegraphics[width=15cm]{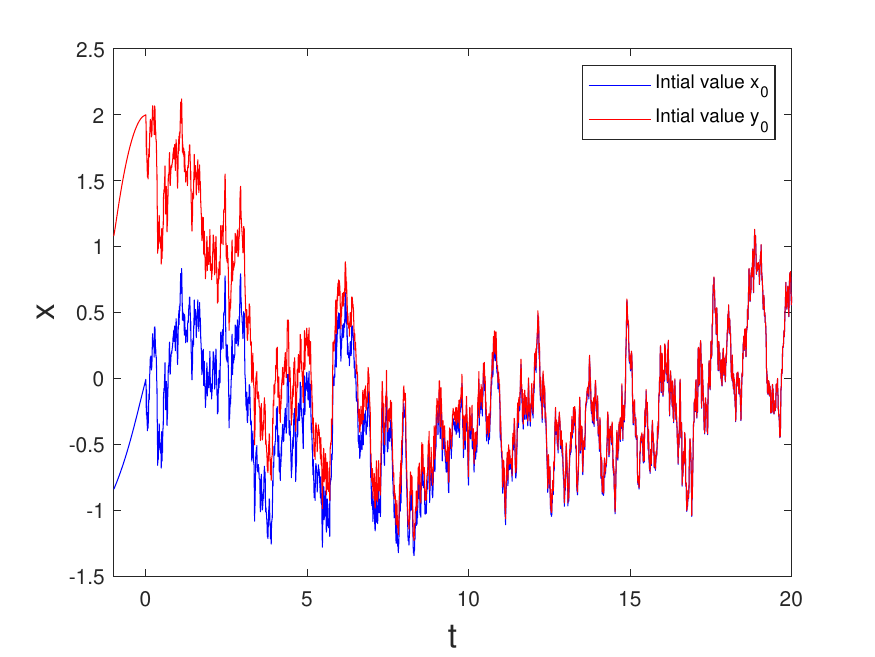}
{\caption{The numerical simulation for $x(t)$ of (\ref{Ex6.2.1}) and (\ref{Ex6.2.2}) with differential initial values $x_0=\sin t$ and $y_0=2\cos t$ for all $t\in[-1,0]$.}}
\end{figure}

\section{Discussion and open problems}
In this paper, we have considered the global stability of stationary solutions for affine and nonlinear SFDEs with additive white noise. In Theorem \ref{thm2}, we assumed that the time delay can not be too large. In fact, in \cite{Lv2}, we have showed that the existence of globally attracting stationary solutions for some semilinear SDDEs with additive white noise can be obtained regardless of the value of $\tau$. Based on some numerical evidences, we leave it as an open problem whether the same conclusion of Theorem \ref{thm2} holds for large delays. In the meantime, the method presented in the proof of Theorem \ref{thm2} may be applied to investigate nonlinear SFDEs satisfying a one-sided Lipschitz condition. Finally, it is left as an open problem to establish similar results in other situations, such as multiplicative noise. These important problems will be the subject of subsequent research.

\end{document}